\documentclass[10pt]{amsart}
\newtheorem{theorem}{Theorem}[section]
\newtheorem{lemma}[theorem]{Lemma}

\newtheorem{corollary}[theorem]{Corollary}
\theoremstyle{definition}
\newtheorem{definition}[theorem]{Definition}

\theoremstyle{remark}
\newtheorem{remark}[theorem]{Remark}

\def\innprod#1#2{\left<#1,#2\right>}
\numberwithin{equation}{section}
\begin{document}
\title{Sharp upperbound and a comparison theorem for the first nonzero Steklov eigenvalue}
\author{Binoy}
\address{Department of Mathematics and Statistics, Indian Institute of Technology, Kanpur 208016, India}
\email{binoy@iitk.ac.in}
\author{G. Santhanam}
\address{Department of Mathematics and Statistics, Indian Institute of Technology, Kanpur 208016, India}
\email{santhana@iitk.ac.in}
\thanks{First author is supported by a research fellowship of CSIR, India.}


\begin{abstract} 
 Let $\mathbb{M}$ denote either a noncompact rank-1 symmetric space $(\overline{M}, ds^2)$ such that $-4 \leq K_{\overline{M}} \leq -1$ 
or a complete, simply connected Riemannian manifold $(\overline{M}, \bar{g})$ 
of dimension $n$ with $K_{\overline{M}} \leq k$ where $k = -\delta^2$ or $0$. Let $\Omega$ be a bounded domain in $\mathbb{M}$ with smooth boundary $\partial \Omega = M$ 
and $\nu_1(\Omega)$ be the first nonzero Steklov eigenvalue on $\Omega$. In the case 
$\mathbb{M} = (\overline{M}, ds^2)$, we prove 
$$
\nu_1(\Omega) \leq \nu_1(B(R))
$$
where $B(R) \subset \overline{M}$ is a geodesic ball such that $Vol(\Omega) = Vol(B(R))$. The equality holds if and only if $\Omega$ is a geodesic ball. 
In the case $\mathbb{M} = (\overline{M}, \bar{g})$, we prove 
$$
\nu_1(\Omega) \leq C_k\nu_1(B_k(R_k))
$$ where $B_k(R_k)$ is a geodesic ball of radius $R_k> 0$ in the simply connected space form $\mathbb{M}(k)$ such that $Vol(\Omega)
 = B_k(R_k)$  and $C_k \geq 1$ is a constant which depends only on the volume of $\Omega$ and the dimension of $\mathbb{M}$.
The inequality is sharp as the equality holds if and only if $\Omega $ is isometric to a geodesic ball in $\mathbb{M}(k)$.
\end{abstract}

\maketitle
\section{Introduction}
Let $\overline{M}$ be an $n$-dimensional complete Riemannian manifold and $\Omega$ be a domain with smooth boundary $M$. 
The Steklov problem is to find a solution of 
\begin{equation}\label{p}
\begin{split}
\Delta f & = 0 \ \ \text{in} \ \Omega  \\ 
\frac{\partial f}{\partial \eta} & =  \nu(\Omega)f \ \  \text{on} \ M
\end{split}
\end{equation} 
where $\eta$ is the normal to $M$ and $\nu(\Omega)$ is a real number. 
This problem was first introduced for bounded domains in the plane by Steklov in 1902 \cite{s}. 
He was motivated by the physical problem of finding the steady state temperature on a bounded planar
 domain such that the flux on the boundary is proportional to the temperature. 
The solution of \eqref{p} in this case represents the temperature (see \cite{b} for other relations to
 physical problems). The problem \eqref{p} occurs  also in
harmonic analysis and inverse problems. Its importance in these areas lies in the fact 
that set of eigenvalues and eigenfunctions of Steklov problem is 
same as that of the well known Dirichlet-Neumann map, which associates to each function on $M$,
the normal derivative of its harmonic extension to $\Omega$ (see for instance \cite{c}).

It is known that the Steklov problem \eqref{p} has a discrete set of eigenvalues
\begin{equation*}
 0 < \nu_1 \leq \nu_2 \leq \nu_3 \leq \cdots \rightarrow \infty.
\end{equation*}
There are several results related to the sharp bounds and comparison of first eigenvalue $\nu_1(\Omega)$. 
For the planar domains, Weinstock \cite{w} proved that 
\begin{quote}
\textit{For all two-dimensional simply connected domains with analytic boundary of given area $A$, circle yields the maximum of $\nu_1$, that is}
$$
\nu_1 \leq \frac{2\pi}{A}.
$$                                                                                                                                                  
\end{quote} 
Hersch and Payne \cite{hp} noticed that Weinstock's proof gives a sharper isoperimetric inequality
$$
\frac{1}{\nu_1} + \frac{1}{\nu_2} \geq \frac{A}{\pi}.
$$
This result was extended to bounded domains in $\mathbb{R}^n$ by F.Brock \cite{fb}. 

In a series of papers J. F. Escobar \cite{e1, e2, e3} studied the geometry and isoperimetric problems of the first non zero eigenvalue of
 the Steklov problem
on a general manifold. In \cite{e3}, author proved the existence of lower bounds for $\nu_1$ under various curvature conditions.
For a bounded simply connected domain $\Omega$ in a 2-dimensional simply connected space form  $M(k),$ the sharp upper bound 
\begin{equation*}
\nu_1(\Omega) \leq \nu_1(B(p,r))
\end{equation*}
was obtained in \cite{e1}, where $B(p,r) \subset  M(k)$ is a geodesic ball of radius $r$ centered at $p$ with  $Area(\Omega) = Area(B(p,r))$. Further 
equality holds only when $\Omega$ is isometric to $B(p,r)$.
The proof of this inequality uses the Weinstock's inequality and the estimates of $\nu_1(B(p,r))$ for the case of two dimensional 
space forms obtained in \cite{e3}. 

In this paper we find the first eigenvalue of geodesic balls in rank-1 symmetric space of all dimensions and
extend the above theorem to noncompact rank-1 symmetric spaces. More precisely we prove following result.
\begin{theorem}\label{thm1}
Let $(\overline{M}, ds^2)$ be a noncompact rank-1 symmetric space with $ -4 \leq K_{\overline{M}} \leq -1$. 
Let $\Omega \subset \overline{M}$ be a bounded domain with smooth boundary $\partial \Omega = M$. Then 
\begin{equation}
 \nu_1(\Omega) \leq \nu_1(B(R))
\end{equation}
where $B(R) \subset \overline{M}$ is a geodesic ball of radius $R>0$ such that $Vol(\Omega) = Vol(B(R))$. 

Further, the equality holds if and only if $\Omega$ is isometric to $B(R)$.
\end{theorem}
Now we move onto the comparison of first nonzero eigenvalue of Steklov problem. 

Let $\Omega $ be a bounded simply connected domain in a complete Riemannian manifold of dimension $2$ with non-positive Gaussian curvature, then 
$$
\nu_1(\Omega) \leq \nu_1(B(R))
$$ 
where $B(R) \subset \mathbb{R}^2$ is a  ball with center at $0$ and radius $R>0$ such that $Vol(\Omega) = Vol(B(R))$. This result was proved 
in \cite{e1} using Weyl's isoperimetric inequality on non-positive curvature manifolds. If $\Omega$ is a geodesic ball, then
following generalization was obtained in \cite{e2};
\begin{equation}\label{ie1}
 \nu_1(B(R)) \leq \nu_1(B_k(R))
\end{equation}
where $B(R), \, R >0$ is a geodesic ball in a complete Riemannian manifold of dimension $2$ or $3$ with the radial curvatures bounded by a constant $k$ 
and $B_k(R)$ is the geodesic ball of radius $R >0$ in the simply connected space form $ \mathbb{M}(k)$ such
that $Vol(B(R)) = Vol(B_k(R))$. Equality in the above inequality holds only if $B(R)$ is isometric to $B_k(R)$. Under some extra condition on 
first eigenvalue of laplacian of the geodesic sphere $S(R)$, above inequality \eqref{ie1} was obtained for arbitrary dimension. Author first proved the 
inequality for arbitrary dimensions and in dimension $2$ and $3$, verified the extra condition imposed on first eigenvalue of laplacian of
the geodesic spheres.  We prove the following similar result  
\begin{theorem}\label{thm2}
Let $(\overline{M}, \bar{g})$ be complete, simply connected manifold of dimension $n$ such that $K_{\overline{M}} \leq k, \, k = -\delta^2$ or $0$,
where $K_{\overline{M}}$ denotes the sectional curvature of $\overline{M}$. Let $\Omega$ be a bounded domain with smooth boundary
$\partial \Omega = M$. Then 
\begin{equation*}
 \nu_1(\Omega) \leq C_k \ \nu_1(B_k(R_k))
\end{equation*}
where $B_k(R_k)$ be a geodesic ball of radius $R_k> 0$ in the simply connected space form $\mathbb{M}(k)$
such that $Vol(\Omega) = B_k(R_k)$ and $C_k \geq 1 $ is a constant which depends only on the volume of $\Omega$ and the dimension of $\mathbb{M}$. 

Further, the equality holds if and only if $\Omega $ is isometric to a geodesic ball in $\mathbb{M}(k)$.
\end{theorem}

We refer to \cite{dc} and \cite{ic} for the basic Riemannian geometry used in this paper.
\section{First eigenvalue of geodesic balls}
Let $(\overline{M}, ds^2)$  denote a rank-1 symmetric space of compact or
noncompact type with dimension, $\text{dim}\overline{M} = kn $ where $k =
\text{dim}_{\mathbb{R}}\mathbb{K}; \ \mathbb{K} = \mathbb{R}, \mathbb{C},
\mathbb{H}$ or $\mathbb{C}a$. For compact type we take the
metric $ds^2$ so that $ 1 \leq K_{\overline{M}} \leq 4 $ and for 
non-compact type we take the metric $ds^2$ so that $ -4 \leq
K_{\overline{M}} \leq -1 $. We denote by $inj(\overline{M}) $, the injectivity radius of $\overline{M}$. For a point $p \in \overline{M}$ 
and $0 < r < inj(\overline{M})$, let $S(r)$ denotes the geodesic sphere
of radius $r$ centered at $p$ and $\Delta_{S(r)}$ denotes the laplacian on
$S(r)$. It is well known [see \cite{gs} for details] that for compact rank-1
symmetric spaces
$$
\lambda_1(S(r)) = \frac{kn-1}{\sin^2r} + \frac{k-1}{\cos^2r}, \ \ \text{for} \ \ 0 < r < \tan^{-1}\left (\sqrt{\frac{kn+1}{k-1}}\right )
$$
and for noncompact rank-1 symmetric spaces 
$$
 \lambda_1(S(r)) = \frac{kn-1}{\sinh^2r} - \frac{k-1}{\cosh^2r} \ \ \ \   \forall \ r > 0
$$
where $ \lambda_1(S(r))$ is the first nonzero eigenvalue of $\Delta_{S(r)}$.

Fix a point $p \in \overline{M}.$ The geodesic polar coordinate system centered at $p$ is denoted by $(r,u)$ where $r> 0$ and $u \in U_p\overline{M}$.
Let $\gamma$ be a geodesic starting at $p$. Then the volume density function $\phi$ along $\gamma$ at the point 
$\gamma(r)$ is given by 
$$
\phi(r) = 
\begin{cases}
 \sin^{kn-1}r\,\cos^{k-1}r & \text{when} \ \overline{M} \ \text{is compact type} \\
 \sinh^{kn-1}r\,\cosh^{k-1}r & \text{when} \ \overline{M} \ \text{is noncompact type}
\end{cases}.
$$ 
Let $S(r)$ denotes concentric spheres with center at $p$ and having radius $r, 0 < r < inj(\overline{M})$. We denote by $A(r)$, the second fundamental form
of $S(r)$. It can be shown that the mean curvature $Tr(A(r))$ of
$S(r)$ is given by $\frac{\phi^{\prime}(r)}{\phi (r)}.$
Let $0 < R < inj(\overline{M})$ be given and $B(R)$ be the geodesic ball of radius $R$ centered at $p$. 

Consider the Steklov eigenvalue problem on $B(R)$: 

Find a solution of the problem 
\begin{equation}\label{pb}
\begin{split}
\Delta f & = 0 \ \ \text{in} \ B(R)  \\ 
\frac{\partial f}{\partial r} & =  \nu(B(R))f \ \  \text{on} \ S(R)
\end{split}
\end{equation} 
where $\nu(B(R))$ is a real number. To find an eigenfunction of \eqref{pb}, first decompose
the laplacian $\Delta$ on $\overline{M}$ along the radial geodesics starting from $p$ as 
$$
\Delta = - \frac{\partial^2}{\partial r^2} - Tr(A(r))\frac{\partial}{\partial r} + \Delta_{S(r)}.
$$
We use the separation of variable technique to find a solution of \eqref{pb}. Consider a smooth function on $B(R)$ given by $h(r,u) = g(r)f(u)$
where $g$ and $f$ are real valued functions defined from $[0,inj(\overline{M}))$ and $U_p\overline{M}$ respectively. Then
\begin{eqnarray}\nonumber
 \Delta h(r, u) & = & - \frac{\partial^2(h(r, u))}{\partial r^2} - Tr(A(r))\frac{\partial (h(r, u))}{\partial r} + \Delta_{S(r)}h(r, u) \\
& = & f(u)\left[-g^{\prime \prime}(r) - Tr(A(r))g^{\prime}(r) \right] + g(r)\Delta_{S(r)}f(u)\nonumber.
\end{eqnarray}
Now suppose that $f$ is an eigenfunction of $\Delta_{S(r)}$ with eigenvalue $\lambda(S(r))$. Then above equation becomes 
\begin{equation*}
 \Delta h(r, u) = f(u)\left[-g^{\prime \prime}(r) - Tr(A(r))g^{\prime}(r) + g(r)\lambda(S(r))\right].
\end{equation*}
Hence we get a solution $h(r, u) = g(r)f(u)$ of \eqref{pb}, where $g$ satisfies 
\begin{equation}\label{e3}
\begin{gathered}
 g^{\prime \prime}(r) + Tr(A(r))g^{\prime}(r) - \lambda(S(r))g(r) = 0, \ \ r \, \in (0, R), \\ 
  g(0) = 0 \ \ \text{and} \ \ g^{\prime}(R) =  \nu(B(R)) g(R).
\end{gathered} 
\end{equation}
Next theorem shows that by taking eigenfunctions corresponding to first eigenvalue $\lambda_1(S(r)) $ of $\Delta_{S(r)}$ 
and $g$ to be corresponding solution of \eqref{e3} we can find the value of $\nu_1(B(R))$.

Before stating the theorem we remark that our notations are similar to the proof given in \cite{e2},
which gives the first nonconstant eigenfunction corresponding to $\nu_1(B(R))$, where $B(R)$ is a geodesic ball of radius $R > 0 $ 
in the Euclidean space $\mathbb{R}^n$ with a rotationally symmetric metric. Also in \cite{e3} the value of $\nu_1(B(R))$ 
is obtained for two dimensional sphere of constant curvature $+1$ and two dimensional hyperbolic space. 
The computation there uses relation between the first nonzero eigenvalues of Steklov problem for conformally
related metrics. Our result gives a rather easy way to find $\nu_1(B(R))$ of geodesic balls in all rank-1 symmetric spaces of any dimension. 
We also remark that arguments given in \cite{e2} and \cite{e3} are analytical in nature where as our arguments are more geometrical. 
\begin{theorem}\label{thm3}
Let $(\overline{M}, ds^2)$ be a rank-1 symmetric space and $B(R)$ be a geodesic ball centered at a point $p \in \overline{M}$
with radius $R$ such that $0< R < inj(\overline{M})$. Then the first non zero eigenvalue $\nu_1(B(R))$ of the Steklov problem on $B(R)$ is given by
\begin{equation*}
 \nu_1(B(R)) = \frac{\int_{B(p, R)}\left(g^2\lambda_1(S(r)) + \left(g^{\prime}\right)^2\right)}{g^2(R)Vol(S(R))}.
\end{equation*}
where $g$ is the radial function satisfying
\begin{equation}\label{e4}
\begin{gathered}
g^{\prime \prime}(r) + Tr(A(r))g^{\prime}(r) - \lambda_1(S(r))g(r) = 0, \ \  r \, \in (0, R), \\
g(0) = 0 \ \ \text{and} \ \ g^{\prime}(R) =  \nu_1(B(R)) g(R).
\end{gathered}
\end{equation}
\end{theorem}
\begin{proof}
We use the following variational characterization to estimate $\nu_1(B(R))$ and corresponding eigen functions
$$
 \nu_1(B(R)) = min\left\{ \frac{\int_{B(R)}\parallel \nabla h\parallel^2}{\int_{S(R)} h^2} \ \big{\vert} \ \int_{S(R)} h = 0 \right\}.
$$
First observe that the space $L^2(B(R))$ is equal to $L^2(0, R) \times L^2(S)$, where $S$ is the unit sphere in $T_p\overline{M}$. 
Let $\{e_i\}_{i=1}^{\infty}$ be a complete orthogonal set of eigenfunctions for the laplacian
$\Delta_{S(r)}$ of $S(r)$ with associated eigenvalues $\lambda_i(S(r))$ such that $ 0 = \lambda_0(S(r))
< \lambda_1(S(r)) \leq \lambda_2(S(r)) \cdots$. Observe that we can choose these functions such that they are constant along radial directions.
For $i \geq 1$, let  $g = g_1, g_2, \cdots $ be functions satisfying
\begin{equation*}
g_i^{\prime \prime}(r) + Tr(A(r))g_i^{\prime}(r) - \lambda_i(S(r))g_i(r) = 0 \ \ r\, \in (0, R)
\end{equation*}
$$
g_i(0) = 0, \ \ \ \ g_i^{\prime}(R) = \beta_ig_i(R).
$$
Let $h_0 = 1 $ and $h_i(r,s) = g_i(r)e_i(u)$. Then the set $\{h_i\}_{i=1}^{\infty}$ is an orthogonal basis for $L^2(B(R))$.
Notice that $\Delta h_i = 0 $ on $B(R)$ for all $ i \geq 1$ and
$$
\int_{S(R)}h_i =  g_i\int_{S(R)}e_i = \frac{g_i}{\lambda_i(S(R))}\int_{S(R)}\Delta_{S(R)}e_i = 0.
$$
Now
\begin{eqnarray*}
 \int_{B(R)}\parallel \nabla h_i\parallel^2 & = & -\frac{1}{2}\int_{B(R)}\Delta(h_i^2)\\
 & = & \int_{S(R)} h_i\innprod{\nabla h_i}{\partial r}.
 \end{eqnarray*}
As $\innprod{\nabla e_i}{\partial r} = 0 $, we get
\begin{equation*}
 \int_{B(R)}\parallel \nabla h_i\parallel^2 = g_i(R)g_i^{\prime}(R)\int_{S(R)} e_i^2.
\end{equation*}
Thus,
\begin{equation}\label{e5}
\frac{\int_{B(R)}\parallel \nabla h_i\parallel^2}{\int_{S(R)} h_i^2} = \frac{g_i(R)g_i^{\prime}(R)\int_{S(R)} e_i^2}{g_i^2(R)\int_{S(R)} e_i^2}
= \frac{g_i^{\prime}(R)}{g_i(R)} = \beta_i.
\end{equation}
But we also have 
\begin{eqnarray*}
 \parallel \nabla h_i\parallel^2 & = & \parallel \nabla^{S(r)} h_i\parallel^2 + \innprod{\nabla h_i}{\partial r}^2 \\
& = & g_i^2\parallel \nabla^{S(r)} e_i\parallel^2  + e_i^2\,(g_i^{\prime})^2 \\
& = & -\frac{1}{2}g_i^2 \Delta_{S(r)}(e_i^2) + \lambda_i(S(r))g_i^2\,e_i^2 + e_i^2\,(g_i^{\prime})^2 \\
& = &  g_i^2\left(\lambda_i(S(r))e_i^2 -\frac{1}{2}\Delta_{S(r)}(e_i^2)\right)  + e_i^2\,(g_i^{\prime})^2\\
& = &  g_i^2\left(\lambda_i(S(r))e_i^2 -\frac{1}{2}\Delta(e_i^2)\right)  + e_i^2\,(g_i^{\prime})^2.
\end{eqnarray*}
The last equality follows as $e_i$'s are constant along the radial directions. Substituting this in \eqref{e5} we get
\begin{eqnarray*}
 \beta_i & = & \frac{\int_{B(R)}\parallel \nabla h_i\parallel^2}{\int_{S(R)} h_i^2} \\
& = & \frac{\int_{B(R)}\left [g_i^2\left(\lambda_i(S(r))e_i^2 -\frac{1}{2}\Delta(e_i^2)\right) + e_i^2\,(g_i^{\prime})^2\right]}{\int_{(R)} h_i^2}.
\end{eqnarray*}
But 
\begin{equation*}
-\frac{1}{2}\int_{B(R)}\Delta(e_i^2) = \int_{S(R)}e_i\innprod{\nabla e_i}{\partial r} = 0.
\end{equation*}
Hence we have 
\begin{equation*}
 \beta_i = \frac{\int_{B(R)}\left [g_i^2\lambda_i(S(r))e_i^2  + (g_i^{\prime})^2e_i^2\right]}{\int_{S(R)} h_i^2}.
\end{equation*}
Since $\lambda_i(S(r)) \geq \lambda_1(S(r))$ for $i \geq 1$, we get that $\beta_i \geq \beta_1$. 
Also all admissible functions in the variational characterization 
of $\nu_1(B(R))$ are orthogonal to constant functions on $S(R)$. Hence we get that $ \nu_1(B(R)) = \beta_1 $. 

Now consider the geodesic normal coordinate system $X = (x_1, x_2, \cdots, x_{kn})$ center at $p$. Then the functions $\frac{x_j}{r}, j= 1, 2, \cdots, kn$
are eigenfunctions corresponding to the first eigenvalue $\lambda_1(S(r))$ of the geodesic sphere $S(r)$ (see \cite{gs} for details). 
Thus from above equation we get
\begin{equation*}
 \nu_1(B(R))\int_{S(R)} g^2\frac{x_j^2}{r^2} = \int_{B(R)}\left [g^2\lambda_1(S(r))\frac{x_j^2}{r^2}  + (g^{\prime})^2\frac{x_j^2}{r^2}
\right].
\end{equation*}
Summation over $j = 1,2, \cdots , kn$ gives
\begin{equation}\label{rs}
 \nu_1(B(R)) = \frac{\int_{B(p, R)}\left(g^2\lambda_1(S(r)) + \left(g^{\prime}\right)^2\right)}{g^2(R)Vol(S(R))}.
\end{equation} 
\end{proof}
\begin{remark}\label{r4}
Notice that the estimate 
\begin{equation*}
 \beta_i = \frac{\int_{B(R)}\left [g_i^2\lambda_i(S(r))e_i^2  + (g_i^{\prime})^2e_i^2\right]}{\int_{S(R)} h_i^2}.
\end{equation*}
is true for any manifold in which eigenfunctions of geodesic spheres are constant along the radial directions.
In particular for $B(R) \subset \mathbb{R}^n$, as the functions $\frac{x_j}{r}, j= 1, 2, \cdots, n$
are eigenfunctions corresponding to the first eigenvalue $\lambda_1(S(r))$ of the geodesic sphere $S(r)$, we see that equation \eqref{rs}
is also valid. A straight forward computation then shows that $\nu_1(B(R)) = \frac{1}{R}$.
\end{remark}
When dimension of $\overline{M} = 2,$ and $(\overline{M}, ds^2)$ is a sphere of constant curvature $+1$ or a hyperbolic space, we have 
\begin{corollary}\label{cor1}
Let $(\overline{M}, ds^2)$ be a two dimensional sphere of constant curvature $+1$ or a two dimensional hyperbolic space and
let $p \in \overline{M}$. For $ 0<R<inj(\overline{M})$, consider the Steklov problem on the geodesic ball $B(R)$. 
\begin{enumerate}
 \item If $(\overline{M}, ds^2)$ is sphere, then 
$$
\nu_1(B(R)) = \frac{1}{\sin R}.
$$
\item If $(\overline{M}, ds^2)$ is hyperbolic space, then 
$$
\nu_1(B(R)) = \frac{1}{\sinh R}.
$$
\end{enumerate}
\end{corollary}
\begin{proof}
If $(\overline{M}, ds^2)$ is sphere, we know that $\lambda_1(S(r)) = \frac{1}{\sin\,r}$ and $g(r) = \tan\frac{r}{2}$. An easy computation then
shows that $\nu_1(B(R)) = \frac{1}{\sin R}.$ The case of hyperbolic space is dealt similarly.
\end{proof}
\section{Proof of theorem \ref{thm1} and theorem \ref{thm2}}
We recall the notion of center of mass which is needed to prove our results.
 
Let $\overline{M}$ be a $n$ dimensional complete Riemannian manifold. For a point $p \in \overline{M}$, we denote by $c(p)$ the
convexity radius of $\overline{M}$ at $p$.
For a subset $A \subset B(q,c(q)) $, for $q \in \overline{M}$, we let $CA$ denote the convex hull of $A$. Let $exp_q:T_q\overline{M}
\rightarrow \overline{M}$ be the exponential map and
$ X = (x_1, x_2, ... , x_n)$ be the normal coordinate system at $q$. We identify $CA$ with $exp_q^{-1}(CA)$ and denote $g_q(X,X)$ as 
$\parallel \! X \! \parallel_q^2$ for $X \in T_q\overline{M}$. Following lemma gives the existence of a center of mass of any measurable subset 
of $\overline{M}$.
\begin{lemma}\label{l1}
Let $A$ be a measurable subset of $\overline{M}$ contained in $B(q_0, c(q_0))$ for some point $q_0 \in \overline{M}$. 
Let $G: [0, 2c(q_0)] \rightarrow \mathbb{R}$ be a 
continuous function such that $G$ is positive on $(0, 2c(q_0))$. Then there exists a point $p \in CA$ such that 
$$
\int_AG(\parallel \!X \!\parallel_p)XdV = 0
$$
where $X = (x_1, x_2, ... , x_n)$ is a geodesic normal coordinate system at $p$.
\end{lemma}
For a proof see \cite{eh} or \cite{gs}.
\begin{definition}
The point $p$ in the above theorem is called a center of mass of the measurable subset $A$ with respect to the mass distribution function $G$.
\end{definition}
Let $\mathbb{M}$ denote either a noncompact rank-1 symmetric space $(\overline{M}, ds^2)$ or a complete, 
simply connected Riemannian manifold $(\overline{M}, \bar{g})$ 
of dimension $n$ with $K_{\overline{M}} \leq k, $ where $k = -\delta^2$ or $0$. Let $\Omega$ be a bounded domain in $\mathbb{M}$ 
with smooth boundary $\partial \Omega = M$. We 
use the following variational characterization to estimate $\nu_1(\Omega)$.
\begin{equation}\label{e6}
\nu_1(\Omega) = min\left\{ \frac{\int_{\Omega}\parallel \nabla h\parallel^2}{\int_M h^2} \ \big{\vert} \ \int_M h = 0 \right\}.
\end{equation}

First consider the case $\mathbb{M} = (\overline{M}, ds^2)$. Recall that in solving Steklov problem on geodesic balls we obtained following equation 
\begin{equation*}
g^{\prime \prime}(r) + Tr(A(r))g^{\prime}(r) - \lambda_1(S(r))g(r) = 0.
\end{equation*}
Using the fact $-\lambda_1(S(r)) = Tr(A)^{\prime}(r)$ (see \cite{gs} for details) we rewrite above equation in the form 
\begin{equation*}
 g^{\prime \prime} + \left(Tr(A)g\right)^{\prime} = 0.
\end{equation*}
This implies that  
\begin{equation*}
 g^{\prime} + Tr(A)g = 1.
\end{equation*}
Since $Tr(A) = \frac{\phi^{\prime}(r)}{\phi (r)},$ above equation can be written as $(g\phi)^{\prime}(r) = \phi (r)$. Thus we get 
\begin{equation}\label{e9}
 g(r) = \frac{1}{\phi (r)}\int_0^r \phi (t)\,dt.
\end{equation}
Let $p \in \mathbb{M}$ be a center of mass of $M$ corresponding to the functions
$g$ and $\frac{1}{r}$. Let $g_i = g.\frac{x_i}{r} $ where $(x_1, \dots , x_{kn})$ is the geodesic normal coordinate system centered at $p$.
Then by lemma \ref{l1}, it follows that $\int_M g_i = 0 $ for $ 1 \leq i \leq kn$. Hence by \eqref{e6} we get 
\begin{equation}\label{e7}
 \nu_1(\Omega)\int_M \sum_{i=1}^{kn}g_i^2 \,dm \leq \int_{\Omega} \sum_{i=1}^{kn}\parallel \nabla g_i \parallel^2 dV.
\end{equation}
But $ \sum_{i=1}^{kn}g_i^2 = g^2$ and 
\begin{eqnarray*}
\sum_{i=1}^{kn}\parallel \nabla g_i \parallel^2  &  = & g^2\sum_{i=1}^{kn}\parallel \nabla \left(\frac{x_i}{r}\right) \parallel^2 + \left(g^{\prime}\right)^2\sum_{i=1}^{kn}\frac{x_i^2}{r^2} 
+ 2\,g\,g^{\prime}\sum_{i=1}^{kn}\frac{x_i}{r}\innprod{\nabla\frac{x_i}{r}}{\partial r} \\
& = & g^2\sum_{i=1}^{kn}\parallel \nabla \left(\frac{x_i}{r}\right) \parallel^2 + \left(g^{\prime}\right)^2 + 
g\,g^{\prime}\innprod{\nabla\left(\sum_{i=1}^{kn}\frac{x_i^2}{r^2}\right)}{\partial r}\\ 
& = & g^2\sum_{i=1}^{kn}\parallel \nabla \left(\frac{x_i}{r}\right) \parallel^2 + \left(g^{\prime}\right)^2.
\end{eqnarray*}
Substituting these into \eqref{e7}, we have
\begin{equation}\label{e8}
 \nu_1(\Omega)\int_M g^2 \, dm \leq \int_{\Omega} \left(g^2\sum_{i=1}^{kn}\parallel \nabla\left(\frac{x_i}{r}\right) \parallel^2 + \left(g^{\prime}\right)^2\right)dV.
\end{equation}
Now $\parallel \nabla\left(\frac{x_i}{r}\right) \parallel^2 = \frac{x_i}{r}\Delta_{S(r)}\left(\frac{x_i}{r}\right) - \Delta_{S(r)}\left(\frac{x_i}{r}\right)^2$. It is well known
that $\Delta_{S(r)}\left(\frac{x_i}{r}\right) = \lambda_1(S(r))\frac{x_i}{r}$. Hence, 
\begin{equation*}
\sum_{i=1}^{kn}\parallel \nabla \left(\frac{x_i}{r}\right) \parallel^2  =\lambda_1(S(r)). 
\end{equation*}
Substituting this in \eqref{e8}, we get 
\begin{equation}\label{e10}
 \nu_1(\Omega)\int_M g^2 \, dm \leq \int_{\Omega}\left(g^2\lambda_1(S(r)) + \left(g^{\prime}\right)^2\right)dV.
\end{equation}

Next consider the case $\mathbb{M} = (\overline{M}, \bar{g})$ with $K_{\overline{M}} \leq k, $ where $k = -\delta^2$ or $0$.
For $r\geq 0$, let  
$$
\sin_{\delta}r = 
\begin{cases}
\frac{1}{\delta}\sinh\delta \,r & if \, K_{\overline{M}} \leq -\delta^2  \\
 r & if \, K_{\overline{M}} \leq 0
 \end{cases} 
$$ 
and 
$$
g_{\delta}(r) = \frac{1}{\sin_{\delta}^{n-1}r}\int_0^r\sin_{\delta}^{n-1}t\,dt. 
$$
Let $p \in \mathbb{M}$ be a center of mass of $M$ corresponding to the functions
$g_{\delta}$ and $\frac{1}{r}$. Let $g_i = g_{\delta}\frac{x_i}{r} $ where $(x_1, \dots , x_n)$ is the geodesic normal coordinate system centered at $p$.
Then by lemma \ref{l1}, it follows that $\int_M g_i = 0 $ for $ 1 \leq i \leq n$. Hence by \eqref{e6} we get 
\begin{equation}\nonumber
 \nu_1(\Omega)\int_M \sum_{i=1}^ng_i^2 \,dm \leq \int_{\Omega} \sum_{i=1}^n\parallel \nabla g_i \parallel^2 dV.
\end{equation}
But $\sum_{i=1}^ng_i^2 = g_{\delta}^2$. Substituting this in above inequality, we get 
\begin{eqnarray}\nonumber
 \nu_1(\Omega)\int_M g_{\delta}^2 \,dm  &\leq  & \int_{\Omega} \left(g_{\delta}^2\sum_{i=1}^n\parallel \nabla \left(\frac{x_i}{r}\right) \parallel^2 +
\left(g_{\delta}^{\prime}\right)^2\right)dV \\ 
& = & \int_{\Omega} \left(g_{\delta}^2\sum_{i=1}^n\parallel \nabla^{S(r)} \left(\frac{x_i}{r}\right) \parallel^2 + \left(g_{\delta}^{\prime}\right)^2\right)dV \label{e13}. 
\end{eqnarray}

To prove the theorems, we will estimate $ \int_M g^2 \,dm, \int_Mg_{\delta}^2dm $ and the right hand side integrals of equations \eqref{e10} and \eqref{e13}.
We fix some notations before proceeding further. First notice that $M$ is a closed hypersurface of $\mathbb{M}$.
The geodesic polar coordinate system centered at $p$ is denoted by $(r,u)$ where $r> 0$ and $u \in U_p\mathbb{M}$.
For any $q \in M$, let $\gamma_q$ be the unique unit speed geodesic segment joining $p$ and $q$ with $ \gamma_q^{\prime}(0) = u.$ 
We write $d(p,q)$ as $t_q(u)$. 

For the case $\mathbb{M} = (\overline{M}, \bar{g})$, let $W \subset T_p\mathbb{M}$ such that $\Omega = exp_p(W)$. Denote 
by $\mathbb{M}(k)$, the simply connected $n$-dimensional space form of constant curvature $k$, where $ k = -\delta^2$ or $0$. Fix a point $p_k \in \mathbb{M}(k)$ 
and an isometry $i:T_p\mathbb{M} \rightarrow T_{p_k}\mathbb{M}(k)$. Let $\Omega_k = exp_{p_k}(i(W)),M_k = 
\partial \Omega_k$ and for $\bar{q} \in M_k,$ we write $d(p_k, \bar{q}) = t_{\bar{q}}(\bar{u})$ where $\bar{u}$ is the tangent 
at $p_k$ of the unit speed geodesic segment $\gamma_{\bar{q}}$ joining between $p_k $ and $\bar{q}$. Also denote by $\phi$ and $\phi_{\delta},$
the volume density functions of $\mathbb{M}$ and $\mathbb{M}(k)$ respectively along radial geodesics starting from $p$ and $p_k$. 

Observe that for any $q \in M$, the geodesic segment joining $p$ and $q$ may intersect $M$ at points other than $q$. 
For $u \in U_p\mathbb{M}$, let 
\begin{equation*}
r(u) = max\{r>0 \, | \, exp_p(ru) \in M\}
\end{equation*} 
and define 
$$
A = \{exp_p(r(u)u) \, | \, u \in U_p\mathbb{M} \}.
$$
Then $A \subset M$ and hence for any nonnegative measurable function $f$ on $M$, we have $\int_Mf \geq \int_A f$.  

Next lemma gives estimates of $\int_M g^2 \, dm$ and $\int_Mg_{\delta}^2dm $.
\begin{lemma}\label{l2}
Let $\Omega \subset \mathbb{M}$ be a bounded domain with smooth boundary $\partial \Omega = M$.
Fix a point $p \in \Omega$. Then the following holds:
\begin{enumerate}
\item $\mathbb{M} = (\overline{M}, ds^2):$\\ Let $g$ be the function defined by \eqref{e9}. Then 
\begin{eqnarray}\label{el1}
\int_M g^2d(p,q)dm \geq Vol(S(p, R))g^2(R)
\end{eqnarray}
where $dm$ is the measure on $M, \, S(p, R)$ is the geodesic sphere and $ B(p, R)$ is the geodesic ball of radius $R$ centered at $p$ in $\mathbb{M}$ 
and $R>0$ is such that $Vol(\Omega) = Vol(B(p, R))$.

The equality holds if and only if $M$ is a geodesic sphere centered at $p$ of radius $R$.  
\item $\mathbb{M} = (\overline{M}, \bar{g}):$\\ Let $g_{\delta}(r) = \frac{1}{\sin_{\delta}^{n-1}r}\int_0^r\sin_{\delta}^{n-1}t\,dt$. Then  
\begin{eqnarray}\label{el2}
\int_M g_{\delta}^2\,d(p,q)dm \geq Vol(S_k(R_k^{'}))g_{\delta}^2(R_k^{'})
\end{eqnarray}
where $dm$ is the measure on $M, \, S_k(R_k^{'})$ is the geodesic sphere and $ B_k(R_k^{'})$ is the
geodesic ball of radius $R_k^{'}$ in $\mathbb{M}(k)$ and $R_k^{'}>0$ is such that $Vol(\Omega_k) = Vol(B_k(R_k^{'}))$.

Further, the equality holds if and only if $M$ is a geodesic sphere in $\overline{M}$ and $\Omega $ is isometric to $ B_k(R_k^{'})$.
\end{enumerate}
\end{lemma}
\begin{proof}
For $q \in M$, let $\phi (t_q(u))$ be the volume density of the geodesic sphere $S(p, t_q(u))$ at the point $q$.
Let $\theta (q)$ be the angle between the unit normal $\eta (q)$ to $M$ and the radial vector $\partial r(q)$. 
Let $du$ be the spherical volume density of the unit sphere $U_p\mathbb{M}$. Then it is known that (\cite{mtw}, p.385, or \cite{tf}, p.1097)
$dm(q) = \sec\theta (q)\phi(t_q(u))du$. 

First consider $\mathbb{M} = (\overline{M}, ds^2)$. In this case we have  
\begin{eqnarray*}
\int_M g^2(d(p,q))dm(q) & \geq & \int_A g^2(d(p,q))dm(q) \\
& = & \int_{U_p\mathbb{M}}g^2(t_q(u)) \, \sec\theta (q)\,\phi(t_q(u))du\\ 
& \geq & \int_{U_p\mathbb{M}} g^2(t_q(u)) \,\phi(t_q(u))du\\
& = & \int_{U_p\mathbb{M}}\int_0^{t_q(u)}\left(2g\,g^{\prime} + g^2\frac{\phi^{\prime}}{\phi}\right)\phi(r) dr\, du \\
& \geq & \int_{\Omega} \left(2g\,g^{\prime} + g^2\frac{\phi^{\prime}}{\phi}\right) \, dV.
\end{eqnarray*}
Notice that $\frac{\phi^{\prime}}{\phi} = Tr(A)$ and 
$g^{\prime} = 1 - Tr(A)g$. Thus $ 2g\,g^{\prime} + g^2\frac{\phi^{\prime}}{\phi} = 2g - Tr(A)g^2$. Now
\begin{eqnarray*}
 \left(2g - Tr(A)g^2\right)^{\prime} & = & -g^2 Tr(A)^{\prime} + 2\left(1 - g\,Tr(A)\right)^2 \\
& = & g^2 \lambda_1(S(r)) + 2\left(1 - g\,Tr(A)\right)^2 \\
& > & 0. 
\end{eqnarray*}
This shows that the function $f(r) = \left(2g\,g^{\prime} + g^2\frac{\phi^{\prime}}{\phi}\right)(r)$ is increasing for $r\,\geq\,0.$
Let $R>0$ be such that $Vol(\Omega) = Vol(B(p,R))$. Then
$$Vol(\Omega \backslash(\Omega \cap B(p,R))) = Vol(B(p, R)\backslash (\Omega \cap B(p, R))).
$$ 
Using these we get, 
\begin{eqnarray*}
\int_{\Omega} f(r)\,dV & = & \int_{\Omega\cap B(p, R)} f(r)\,dV + 
\int_{\Omega \backslash(\Omega\cap B(p, R))} f(r)\,dV \\
& = & \int_{B(p, R)}f(r)\,dV - \int_{B(p, R)\backslash \Omega\cap B(p, R)} f(r)\,dV \\
& & + \int_{\Omega\backslash(\Omega \cap B(p, R))} f(r)\,dV \\
& \geq & \int_{B(p, R)}f(r)\,dV - \int_{B(p, R)\backslash \Omega\cap B(p, R)} f(r)\,dV\\
& & +\int_{\Omega\backslash(\Omega \cap B(p, R))} f(R)\,dV \\
& = & \int_{B(p, R)}f(r)\,dV + 
\int_{B(p, R)\backslash \Omega\cap B(p, R)} (f(R)- f(r))\,dV  \\ 
& \geq & \int_{B(p, R)}f(r)\,dV \\
& = & \int_{U_p\overline{M}}\int_0^R \left(2g\,g^{\prime} + g^2\frac{\phi^{\prime}}{\phi}\right)\phi(r)\,dr\,du \\
& = & \int_{U_p\overline{M}} g^2(R)\,\phi(R) \,du\\ 
& = & g^2(R)\,\phi(R)\int_{U_p\overline{M}} du\\ 
& = & Vol(S(p,R))g^2(R).
\end{eqnarray*}
Further equality holds in above equation if and only if following conditions hold:
\begin{itemize}
\item 
$\sec \, \theta(q) = 1$ for all points $q \in M.$
\item 
$Vol(B(p, R))\backslash (\Omega\cap B(p, R))) = 0$.
\end{itemize}
Now $\sec \, \theta(q) = 1$ implies that the normal $\eta(q) = \partial r(q).$ Thus first condition implies that $\eta(q) = \partial r(q)$ 
for all points in $q \in M$. This shows that $M$ is a geodesic sphere centered at $p$ and hence $\Omega = B(p, R)$. 

Next consider the case $\mathbb{M} = (\overline{M}, \bar{g})$. Recall that in this case $g_{\delta}(r) = \frac{1}{\sin_{\delta}^{n-1}r}\int_0^r\sin_{\delta}^{n-1}t\,dt$. 
For $q \in M,$ consider the corresponding geodesic segment $\gamma_{\bar{q}}$ joining $p_k $ and $\bar{q} \in M_k$ in $\mathbb{M}(k)$.
Then by Rauch comparison theorem \cite{dc} it follows that $l(\gamma_q) \geq l(\gamma_{\bar{q}})$ and 
hence $t_q(u) \geq t_{\bar{q}}(\bar{u})$. By Gunther's volume comparison theorem \cite{ghl} 
we also have $\phi (t_q(u)) \geq \phi_{\delta}(t_q(u)) = \sin_{\delta}^{n-1}t_q(u)$ along the geodesics $\gamma_p $ and $\gamma_{\bar{q}}$ respectively.
Hence, 
\begin{eqnarray*}\nonumber
\int_M g_{\delta}^2(d(p,q))dm(q) & \geq & \int_A g_{\delta}^2(d(p,q))dm(q) \\
& = & \int_{U_p\mathbb{M}}g_{\delta}^2(t_q(u)) \, \sec\theta (q)\,\phi(t_q(u))du\\
& \geq & \int_{U_p\mathbb{M}} g_{\delta}^2(t_q(u)) \,\phi(t_q(u))du \\
& \geq & \int_{U_p\mathbb{M}}g_{\delta}^2(t_q(u)) \,\phi_{\delta}(t_q(u))du \nonumber \\
& \geq & \int_{U_p\mathbb{M}(k)}g_{\delta}^2(t_{\bar{q}}(\bar{u})) \,\phi_{\delta}(t_{\bar{q}}(\bar{u}))d\bar{u} \nonumber\\
& = & \int_{U_p\mathbb{M}(k)} \int_0^{t_{\bar{q}}(\bar{u})}\left(2g_{\delta}\,g_{\delta}^{\prime} + g_{\delta}^2\frac{\phi_{\delta}^{\prime}}{\phi_{\delta}}\right)\phi_{\delta}(r) dr\, du \\
& \geq & \int_{\Omega_k} \left(2g_{\delta}\,g_{\delta}^{\prime} + g_{\delta}^2\frac{\phi_{\delta}^{\prime}}{\phi_{\delta}}\right) \, dV.\nonumber
\end{eqnarray*}
As earlier, the function $f(r) = \left(2g_{\delta}\,g_{\delta}^{\prime} + g_{\delta}^2\frac{\phi_{\delta}^{\prime}}{\phi_{\delta}}\right)(r)$
is increasing for $r > 0$ and hence proceeding similarly we get
\begin{eqnarray*}
\int_M g_{\delta}^2\,d(p,q)dm \geq Vol(S_k(R_k^{'}))g_{\delta}^2(R_k^{'})
\end{eqnarray*}
where $S_k(R_k^{'})$ is a geodesic sphere and $B_k(R_k^{'})$ is a geodesic ball in the space form $\mathbb{M}(k)$ and $R_k^{'}>0$ is such that $Vol(\Omega_k) = Vol(B_k(R_k^{'})$.
 
Further, equality holds in above inequality if and only if following conditions hold:
\begin{itemize}
\item 
$\sec \, \theta(q) = 1$.
\item
$l(\gamma_q) = l(\gamma_{\bar{q}})$  for all points $q \in M, \phi (r)  = \phi_{\delta}(r)$ for $r \, 
\leq \text{diam}(\mathbb{M})$ along the geodesics $\gamma_p $ and $\gamma_{\bar{q}}$ respectively. 
\item 
$Vol(B_k(R_k^{'})\backslash (\Omega_k\cap B_k(R_k^{'}))) = 0$.
\end{itemize}
Now $\sec \, \theta(q) = 1$ implies that the normal $\eta(q) = \partial r(q).$ Thus first condition implies that $\eta(q) = \partial r(q)$ 
for all points in $q \in M$. This shows that $M$ is a 
geodesic sphere centered at $p$. The equality criteria in Gunther's volume comparison theorem (\cite{jhe}, \cite{ghl}) says that if 
$\phi (r)  = \phi_{\delta}(r)$ for $ r \leq R_k^{'} \leq \text{diam}(\mathbb{M})$ then the geodesic balls $B(p,R_k^{'})$ and $B_k(R_k^{'})$ are isometric. 
Hence $\Omega $ is isometric to $ B_k(R_k^{'})$.
\end{proof}
\begin{proof}[Proof of theorem \ref{thm1}]
Recall the inequality \eqref{e10} 
\begin{equation}\nonumber
 \nu_1(\Omega)\int_M g^2 \, dm \leq \int_{\Omega}\left(g^2\lambda_1(S(r)) + \left(g^{\prime}\right)^2\right)dV.
\end{equation}
By lemma \ref{l2}, above inequality becomes
\begin{equation}\label{e11}
 \nu_1(\Omega)Vol(S(p, R))g^2(R) \leq \int_{\Omega}\left(g^2\lambda_1(S(r)) + \left(g^{\prime}\right)^2\right)dV.
\end{equation}
Next lemma shows that the function $g^2(r)\lambda_1(S(r)) + \left(g^{\prime}\right)^2(r)$ is decreasing for $ r > 0$.
\begin{lemma}\label{l3}
Let $(\overline{M}, ds^2)$ be a noncompact rank-1 symmetric space and $g$ be the function given by \eqref{e9}. Then 
$g^2(r)\lambda_1(S(r)) + \left(g^{\prime}\right)^2(r)$ is a decreasing function for $r > 0$. 
\end{lemma}
\begin{proof}
First consider $\mathbb{R}\mathbb{H}^n$. In this case 
\begin{equation*}
g(r) = \frac {1}{sinh^{n-1}r}{\int_0}^r sinh^{n-1}t\, dt \ \ \ \text{and} \ \ \ \lambda_1(S(r)) = \frac{n-1}{sinh^2r}.
\end{equation*}
We claim that $g^{\prime}(r) > 0$ for all $r >0$. As $g^{\prime} = 1 - Tr(A)g,$ it follows that $\ g^{\prime} > 0$ if and only if 
\begin{equation*}
sinh^nr \geq (n-1)cosh\,r{\int_0}^r sinh^{n-1}t\, dt \ \ \ \ \text{for all} \ \ r \, > \, 0.
\end{equation*}
Let $h_1 (r) = sinh^nr$ and $h_2 (r) = (n-1)cosh\,r{\int_0}^r sinh^{n-1}t\, dt$. Then the functions $ h_1, h_2 $ are strictly increasing 
and $h_1(0) = h_2(0) = 0 $. Thus to prove above inequality, it is enough to prove that $ {h_1}^{\prime}(r) \geq {h_2}^{\prime}(r) $ for $r > 0 $. Now
${h_1}^{\prime}(r) \geq {h_2}^{\prime}(r) $ if and only if
\begin{equation*}
 \cosh \,r \sinh^{n-2}r \geq (n-1)\int_0^r \sinh^{n-1}t\,dt \ \ \ \ \text{for all} \ \ r \, > \, 0.
\end{equation*}
Performing one more step in a similar way we see that the above inequality is true if and only if 
$$
\cosh^2r\,\sinh^{n-3}r \geq \sinh^{n-1}r \ \ \ \ \text{for all} \ \ r \, > \, 0
$$
which is true. Hence the function $g^{\prime}(r) > 0$ for all $r>0$.
Next, we show that $g^{\prime \prime}(r) \leq 0$ for all $r > 0$. This holds if and only if 
\begin{equation*}
sinh^nr\,cosh\,r \geq [n +(n-1)sinh^2r]{\int_0}^r sinh^{n-1}t\, dt \ \ \ \ \text{for all} \ \ r \, > \, 0.
\end{equation*}
Let $g_1(r) = sinh^nr\,cosh\,r$ and $ g_2(r) = [n +(n-1)sinh^2r]{\int_0}^r sinh^{n-1}t\, dt $. These functions are strictly increasing 
and $g_1(0) = g_2(0) = 0 $. Thus to prove above inequality it is enough to prove that $ {g_1}^{\prime}(r) \geq {g_2}^{\prime}(r) $ for $r > 0 $. We see that
${g_1}^{\prime}(r) \geq {g_2}^{\prime}(r) $ if and only if
\begin{equation*}
sinh^nr \geq (n-1)cosh\,r{\int_0}^r sinh^{n-1}t\, dt \ \ \ \ \text{for all} \ \ r \, > \, 0,
\end{equation*}
which is true as we have already seen. Thus the function $\left(g^{\prime}\right)^2$ is decreasing.  \\
Next consider the function $g^2(r)\lambda_1(S(r))$. It is decreasing for $r > 0$ if and only if
\begin{equation*}
sinh^nr \leq n\,cosh\,r{\int_0}^r sinh^{n-1}t\, dt \ \ \ \ \text{for all} \ \ r \, > \, 0.
\end{equation*}
A similar argument as above shows that it is true.

Other cases of noncompact rank-1 symmetric spaces follows quickly as the integral in the definition of $g$ can be computed. We illustrate it by 
explaining the case of $\mathbb{C}a\mathbb{H}^2$. 
In this case we have 
\begin{equation*}
g(r) = \frac{1}{sinh^{15}r \ cosh^7r}{\int_0}^rsinh^{15}t \ cosh^7t dt \ \ \  \text{and} \ \ \ \lambda_1(S(r)) =
\frac{15}{sinh^2r} - \frac{7}{cosh^2r}.
\end{equation*}
By performing the integration we get that 
\begin{equation*}
g(r) = \frac{1}{22}sinh\,r\left (\frac{sech^7r}{240} + \frac{sech^5r}{15} +
\frac{3}{10}sech^3r + sech\,r\right ).
\end{equation*}
A further computation shows that the functions $\left(g^{\prime}\right)^2(r)$ and $g^2(r)\lambda_1(S(r))$ are decreasing for $ r > 0$.
\end{proof}
As the function $ g^2\lambda_1(S(r)) + \left(g^{\prime}\right)^2$ is decreasing, an argument similar to that of in lemma \ref{l2} shows that 
\begin{equation}\label{e20}
\int_{\Omega}\left(g^2\lambda_1(S(r)) + \left(g^{\prime}\right)^2\right)dV \leq \int_{B(p, R)}\left(g^2\lambda_1(S(r)) + \left(g^{\prime}\right)^2\right)dV. 
\end{equation} 
Substituting this in \eqref{e11}, we get 
\begin{equation}\label{e12}
 \nu_1(\Omega) \leq \frac{\int_{B(p, R)}\left(g^2\lambda_1(S(r)) + \left(g^{\prime}\right)^2\right)dV}{g^2(R)Vol(S(R))}.
\end{equation}
By theorem \ref{thm3} we have
\begin{equation*}
 \nu_1(B(R)) = \frac{\int_{B(p, R)}\left(g^2\lambda_1(S(r)) + \left(g^{\prime}\right)^2\right)}{g^2(R)Vol(S(R))}.
\end{equation*}
Substitution of above equality in \eqref{e12} gives the required result 
\begin{equation}\label{e14}
 \nu_1(\Omega) \leq \nu_1(B(R)).
\end{equation}
Furthermore, equality holds if and only if equality criteria in lemma \ref{l2} hold. Thus the equality in \eqref{e14}
holds if and only if $\Omega$ is isometric to the geodesic sphere $B(R)$.
\end{proof}
\begin{remark}
 In the case of compact rank-1 symmetric space, the integrand in the right hand side of integral in \eqref{e10} is strictly increasing. Hence we do not 
have the estimate similar to \eqref{e20}. This shows that the above proof fails for the compact rank-1 symmetric spaces.
\end{remark}
In corollary \ref{cor1} we have seen that when $\overline{M}$ is two dimensional hyperbolic space,
then $\nu_1(B(R))$ is easily computed. Thus we have
\begin{corollary}\label{cor2}
Let $(\overline{M}, ds^2)$ be a two dimensional hyperbolic space and $\Omega$ be a bounded domain with smooth boundary. Then
$$
\nu_1(\Omega) \leq \frac{1}{\sinh R}
$$ where $R>0$ is such that the geodesic ball $B(R)$ has the same volume as that of $\Omega$. 

Further the equality holds if and only if $\Omega$ is isometric to the geodesic ball $B(R)$.
\end{corollary}
\begin{remark}
In \cite{e1}, using Weinstock theorem \cite{w} and two dimensional isoperimetric inequality, author proved 
above theorem for two dimensional simply connected space forms, in particular for two dimensional hyperbolic space which is a 
rank-1 symmetric space of noncompact type. Our result is a generalization to all rank-1 symmetric spaces of noncompact type of all dimension. 
\end{remark}
Observe that the same line of proof works for bounded domains in $\mathbb{R}^n$ with smooth boundary.
Also in remark \ref{r4} we have seen that $\nu_1(B(R)) = \frac{1}{R}$ for $B(R) \subset \mathbb{R}^n$. Thus we have 
\begin{theorem}
Let $\Omega \subset \mathbb{R}^n$ be a bounded domain with smooth boundary $\partial \Omega = M$. Then 
\begin{equation*}
 \nu_1(\Omega) \leq \frac{1}{R}
\end{equation*}
where $ R > 0$ is such that the geodesic ball $B(R)$ has the same volume as that of $\Omega$.

Further, the equality holds if and only if $\Omega $ is isometric to a geodesic ball $B(R)$.
\end{theorem} 
\begin{proof}[Proof of theorem \ref{thm2}]
Recall the inequality \eqref{e13}
\begin{eqnarray*}\nonumber
 \nu_1(\Omega)\int_M g_{\delta}^2 \,dm  \leq  \int_{\Omega} \left(g_{\delta}^2\sum_{i=1}^n\parallel \nabla^{S(r)} \left(\frac{x_i}{r}\right) \parallel^2 +
\left(g_{\delta}^{\prime}\right)^2\right)dV.  
\end{eqnarray*}
By inequality \eqref{el2} in lemma \ref{l2}, we get 
\begin{equation}\label{e15}
 \nu_1(\Omega)Vol(S_k(R_k^{'}))\,g_{\delta}^2(R_k^{'}) \leq \int_{\Omega} \left(g_{\delta}^2\sum_{i=1}^n\parallel \nabla^{S(r)} \left(\frac{x_i}{r}\right) \parallel^2 + 
\left(g_{\delta}^{\prime}\right)^2\right)dV.
\end{equation}
Next lemma gives an estimate of $\sum_{i=1}^n\parallel \nabla^{S(r)} \left(\frac{x_i}{r}\right) \parallel^2 =
\frac{1}{r^2}\sum_{i=1}^n\parallel \nabla^{S(r)} x_i \parallel^2$. 
\begin{lemma}\label{l4}
Let $(\overline{M}, \bar{g})$ be a complete, simply connected Riemannian manifold of dimension $n$ such that the sectional 
curvature satisfies $K_{\overline{M}} \leq k$ where $ k = -\delta^2$ or $0$. Fix a point $p\in \overline{M}$ and let $X = (x_1, x_2, ..., x_n)$ be the geodesic normal
coordinate system at $p$. Denote by $S(r)$, the geodesic sphere of radius $r > 0$ center at $p$. Then
$$
\sum_{i=1}^n\parallel \nabla^{S(r)}x_i \parallel^2  \leq (n-1)\frac{r^2}{\sin_{\delta}^2r}.
$$
\end{lemma}
\begin{proof}
Let $q \in S(r)$ and $(e_1, ... , e_{n-1})$ be an orthonormal basis of $T_qS(r)$. Then
$$
\sum_{i=1}^n\parallel \nabla^{S(r)}x_i \parallel^2  = \sum_{i=1}^n\sum_{l=1}^{n-1} \innprod{\nabla^{S(r)}x_i}{e_l}^2 = 
\sum_{i=1}^n\sum_{l=1}^{n-1} \innprod{\nabla x_i}{e_l}^2.
$$ 
Let $e_l = d(exp_p)\bar{e_l}. $
Note that $\innprod{\nabla x_i}{e_l} = e_l(x_i) = \bar{e_l}(x_i\circ exp_p)$ is the $i$-th component of 
$\bar{e_l}$ in the geodesic normal coordinate at $p$. Thus 
$$
\sum_{i=1}^n\innprod{\nabla x_i}{e_l}^2 = \parallel \bar{e_l} \parallel^2.
$$
Consider a unit speed geodesic $\gamma$ in $\mathbb{M}$ such that $\gamma(0) = p$ and $\gamma(r) = q$. Let 
$J_l$ be the Jacobi field along $\gamma$ such that 
$J_l(0) = 0 $ and $J_l^{\prime}(0) = \bar{e_l}$. Then $e_l = d(exp_p)\bar{e_l} = \frac{1}{r}J_l(r).$ 
Fix a point $p_k \in \mathbb{M}(k)$ and a unit speed geodesic $\bar{\gamma}$ such that $\bar{\gamma}(0) =  p_k$.
Let $\bar{u}$ be a unit vector at $p_k$ and $E(t)$ be the vector 
field obtained by parallel translating $\bar{u}$ along $\bar{\gamma}(t)$. Consider the Jacobi field 
$$
J_{\delta}(t) = \sin_{\delta}t\,\parallel J_l^{\prime}(0) \parallel E(t)
$$ 
along $ \bar{\gamma}$. By the Rauch comparison theorem 
$$
\parallel J_{\delta}(t) \parallel \leq \parallel J_l(t) \parallel \ \ \ \  \text{for} \ \ t > 0.  
$$
Hence $\parallel e_l \parallel^2  = \frac{1}{r^2}\parallel J_l(r) \parallel^2 \geq \frac{1}{r^2} \parallel J_{\delta}(r) \parallel^2 
 = \frac{\sin_{\delta}^2r}{r^2}\parallel J_l^{\prime}(0) \parallel^2,$ which implies
$$
\parallel \bar{e_l} \parallel^2 = \parallel J_l^{\prime}(0) \parallel^2 \leq \frac{r^2}{\sin_{\delta}^2r}.
$$
Thus we get 
$$
\sum_{i=1}^n\parallel \nabla^{S(r)}x_i \parallel^2  = \sum_{i=1}^n\sum_{l=1}^{n-1} \innprod{\nabla x_i}{e_l}^2 
= \sum_{l=1}^{n-1} \parallel \bar{e_l} \parallel^2 \leq (n-1)\frac{r^2}{\sin_{\delta}^2r}.
$$  
\end{proof}
Substitution of above estimate in \eqref{e15} gives
\begin{eqnarray*}
\nu_1(\Omega)Vol(S_k(R_k^{'}))\,g_{\delta}^2(R_k^{'}) \leq \int_{\Omega} \left(\frac{n-1}{\sin_{\delta}^2r}g_{\delta}^2 + \left(g_{\delta}^{\prime}\right)^2\right)dV.
\end{eqnarray*}
Using the fact that $\lambda_1(S_k(r)) = \frac{n-1}{\sin_{\delta}^2r}$, above inequality becomes 
\begin{equation}\label{e16}
\nu_1(\Omega)Vol(S_k(R_k^{'}))\,g_{\delta}^2(R_k^{'}) \leq \int_{\Omega}\left(g_{\delta}^2\lambda_1(S_k(r)) + \left(g_{\delta}^{\prime}\right)^2\right)dV. 
\end{equation}
Let $R_k > 0$ be such that $Vol(\Omega) =Vol(B_k(R_k)),$ where $B_k(R_k)$ is a geodesic ball in $\mathbb{M}(k)$. Consider the ball $B(R_k) = B(p,R_k) \subset \overline{M}.$
Then by Gunther's volume comparison theorem $Vol(\Omega)  \leq Vol(B(R_k))$. By lemma \ref{l3}, the function $g_{\delta}^2\lambda_1(S_k(r)) + \left(g_{\delta}^{\prime}\right)^2$
is decreasing. Also notice that 
$$
Vol\left(\Omega\backslash \left(\Omega \cap B(R_k)\right)\right) \leq Vol\left(B(R_k)\backslash \left(\Omega \cap B(R_k)\right)\right).
$$
Using these facts, an argument similar to that of in lemma \ref{l2} gives
\begin{equation*}
 \int_{\Omega}\left(g_{\delta}^2\lambda_1(S_k(r)) + \left(g_{\delta}^{\prime}\right)^2\right)dV \leq 
\int_{B(R_k)}\left(g_{\delta}^2\lambda_1(S_k(r)) + \left(g_{\delta}^{\prime}\right)^2\right)dV.
\end{equation*}
Hence \eqref{e16} becomes
\begin{eqnarray*}\nonumber
\nu_1(\Omega) & \leq & \frac{\int_{B(R_k)}\left(g_{\delta}^2\lambda_1(S_k(r)) + \left(g_{\delta}^{\prime}\right)^2\right)dV}{Vol(S_k(R_k^{'}))\,g_{\delta}^2(R_k^{'})}\\ \nonumber
& = & \frac{\int_{B(R_k)}\left(g_{\delta}^2\lambda_1(S_k(r)) + \left(g_{\delta}^{\prime}\right)^2\right)dV}
{\int_{B_k(R_k)}\left(g_{\delta}^2\lambda_1(S_k(r)) + \left(g_{\delta}^{\prime}\right)^2\right)dV}
\frac{\int_{B_k(R_k)}\left(g_{\delta}^2\lambda_1(S_k(r)) + \left(g_{\delta}^{\prime}\right)^2\right)dV}{g_{\delta}^2(R_k^{'})\,Vol(S_k(R_k^{'}))}\\ \nonumber
& = & C_k\frac{\int_{B_k(R_k)}\left(g_{\delta}^2\lambda_1(S_k(r)) + \left(g_{\delta}^{\prime}\right)^2\right)dV}{g_{\delta}^2(R_k)\,Vol(S_k(R_k))}\nonumber
\end{eqnarray*}
where 
\begin{eqnarray*}
C_k  & = &  \frac{g_{\delta}^2(R_k)\,Vol(S_k(R_k))}{g_{\delta}^2(R_k^{'})\,Vol(S_k(R_k^{'}))}
\frac{\int_{B(R_k)}\left(g_{\delta}^2\lambda_1(S_k(r)) + \left(g_{\delta}^{\prime}\right)^2\right)dV}
{\int_{B_k(R_k)}\left(g_{\delta}^2\lambda_1(S_k(r)) + \left(g_{\delta}^{\prime}\right)^2\right)dV} \\
& = &  \frac{g_{\delta}^2(R_k)\,\phi_{\delta}(R_k)}{g_{\delta}^2(R_k^{'})\,\phi_{\delta}(R_k^{'})}
\frac{\int_{B(R_k)}\left(g_{\delta}^2\lambda_1(S_k(r)) + \left(g_{\delta}^{\prime}\right)^2\right)dV}
{\int_{B_k(R_k)}\left(g_{\delta}^2\lambda_1(S_k(r)) + \left(g_{\delta}^{\prime}\right)^2\right)dV}.
\end{eqnarray*}
Notice that $C_k \geq 1$ and it depends only upon the volume of $\Omega$ and the dimension of $\mathbb{M}$. In the proof of theorem \ref{thm1}, we have seen that 
\begin{eqnarray*}
 \nu_1(B_k(R_k)) & = & \frac{\int_{B_k(R_k)}\left(g_{\delta}^2\lambda_1(S_k(r)) + \left(g_{\delta}^{\prime}\right)^2\right)dV}{g_{\delta}^2(R_k)Vol(S_k(R_k))}.
\end{eqnarray*} 
This implies,
\begin{equation}\label{e17}
 \nu_1(\Omega) \leq C_k \ \nu_1(B_k(R_k)).
\end{equation}
Suppose now that the equality holds. Then the equality criteria in lemma \ref{l2} holds, which shows that $\Omega$ is isometric to the geodesic ball 
$B_k(R_k^{'})$. This in turn implies that the constant $ C_k = 1$. Thus the equality holds in \eqref{e17}  if and only if $\Omega$ is 
isometric to a geodesic ball in $\mathbb{M}(k)$. 
\end{proof}

\end{document}